\documentclass[bezier,amstex, oneside,reqno]{amsart}
\usepackage{amsmath, amssymb, amsthm, verbatim, euscript}
\usepackage[dvips]{graphicx}
\usepackage{epsfig}
\usepackage{color}
\begingroup\makeatletter\ifx\SetFigFont\undefined%
\gdef\SetFigFont#1#2#3#4#5{%
  \reset@font\fontsize{#1}{#2pt}%
  \fontfamily{#3}\fontseries{#4}\fontshape{#5}%
  \selectfont}%
\fi\endgroup%

\newtheorem*{q}{Question}
\newtheorem{theorem}{Theorem}
\newtheorem{lemma}[theorem]{Lemma}
\newtheorem{prop}[theorem]{Proposition}
\newtheorem{cor}[theorem]{Corollary}
\newtheorem{step}{Step}
\theoremstyle{remark}
\newtheorem*{remark}{Remark}

\theoremstyle{remark}
\newtheorem{definition}[theorem]{Definition}

\begin{document}
\author{ Andrey Gogolev}
\title{Diffeomorphisms H\"older conjugate to Anosov diffeomorphisms}
\begin{abstract}
We show by means of a counterexample that a $C^{1+Lip}$
diffeomorphism H\"older conjugate to an Anosov diffeomorphism is not
necessarily Anosov. Also we include a result from the 2006 Ph.D.
thesis of T.~Fisher: a $C^{1+Lip}$ diffeomorphism H\"older conjugate
to an Anosov diffeomorphism is Anosov itself provided that H\"older
exponents of the conjugacy and its inverse are sufficiently large.
\end{abstract}
\date{August 26, 2008, revised on February 09}
 \maketitle

\section{Introduction}
Consider  Anosov diffeomorphisms $f$ and $g$ of a compact smooth
manifold $M$ that are conjugate by a homeomorphism $h$:
$$
h\circ f=g\circ h.
$$
 It is
well known and easy to show that $h$ is in fact H\"older continuous.
When we say that the conjugacy is H\"older or that two
diffeomorphisms are H\"older conjugate we mean that the conjugacy
and its inverse are H\"older continuous.

It is natural to ask the following converse question.
\begin{q}
Is every diffeomorphism that is H\"older conjugate to an Anosov diffeomorphism itself Anosov?
\end{q}

This question was asked by A.~Katok. His motivation came from
differentiable rigidity of higher rank Anosov actions. For example,
a popular object of study is a $\mathbb Z^k$-action which contains
Anosov elements and which is conjugate to an algebraic action for
which Anosov elements are dense. If the answer to the question above
were positive then we would immediately get that Anosov elements are
dense in the original action. Moreover, the Weyl chamber picture in
$\mathbb R^k$ for non-algebraic action would be the same as the one
for the algebraic action. Normally this information is unavailable
or only available through difficult means otherwise. See upcoming
book~\cite{KN} for an introduction to rigidity of Anosov actions.

Unfortunately the answer is negative. We will provide a concrete
counterexample of a $C^{1+Lip}$ diffeomorphism of the 2-torus
$\mathbb T^2$ H\"older conjugate to Anosov but not Anosov itself. In
fact, the counterexample can be constructed to be $C^r$ for any
$r\in(1,3)$ (see remark after Theorem 1 below).

The basic method to produce a non-Anosov diffeomorphism that is
topologically conjugate to Anosov one is to start with an Anosov
diffeomorphism and isotope it pushing stable eigenvalues at a fixed
point to the unit circle. This can be done so that stable and
unstable foliations persist. They remain mutually transversal
everywhere but not uniformly contracting and expanding. The new
system is topologically conjugate to the original Anosov map.
See~\cite{K} for the detailed construction and the proof. Another
similar example was considered in~\cite{L}. It has an additional
feature: stable and unstable manifolds at the fixed point are
tangent.

Looking at the behavior of orbits approaching a fixed point along
the stable manifold we have an exponentially fast approach for the
Anosov map conjugated to a much slower sub-exponential approach. A
H\"older continuous conjugacy would necessarily preserve the
exponential speed, only changing the exponent. This shows that these
diffeomorphisms are far from being H\"older conjugate. Note that in
the meantime the conjugacy or its inverse may turn out to be even
Lipschitz.

Another way to produce such a diffeomorphism is to start with an
Anosov diffeomorphism and ``bend" unstable manifold of a
heteroclinic point $R$ until stable and unstable manifolds at $R$
become tangent. This isotopy can be done locally in the neighborhood
of $R$. The result is a diffeomorphism with stable and unstable
foliation being transverse everywhere but along the orbit of $R$.
Along this orbit stable and unstable manifolds exhibit a tangency.
If we isotope inside of $\mathrm{Diff}^\infty(M)$ then the tangency
is at least cubic: it cannot be quadratic since the stable and
unstable foliations are topologically transverse. This bifurcation
at the boundary of Anosov systems was independently studied by
H.~Enrich~\cite{E} and Ch.~Bonatti, L.~Diaz, F.~Vuillemin
\cite{BDV}. It was shown in~\cite{E} that the new system is
conjugate to the original Anosov map. All periodic points remain
hyperbolic. Hence, unlike in the previous situation, there is a hope
that the topological conjugacy is in fact H\"older continuous.


\begin{figure}[htbp]
\begin{center}

\begin{picture}(0,0)%
\includegraphics{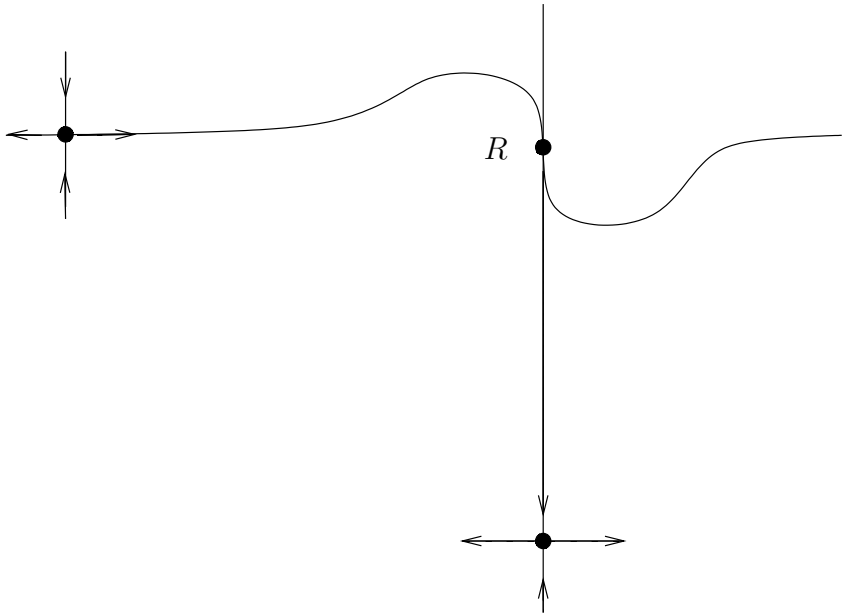}%
\end{picture}%
\setlength{\unitlength}{3947sp}%
\begingroup\makeatletter\ifx\SetFigFont\undefined%
\gdef\SetFigFont#1#2#3#4#5{%
  \reset@font\fontsize{#1}{#2pt}%
  \fontfamily{#3}\fontseries{#4}\fontshape{#5}%
  \selectfont}%
\fi\endgroup%
\begin{picture}(5274,3849)(289,-3223)
\put(3301,-361){\makebox(0,0)[lb]{\smash{{\SetFigFont{12}{14.4}{\rmdefault}{\mddefault}{\updefault}{\color[rgb]{0,0,0}$R$}%
}}}}
\end{picture}%

\end{center}
\caption{Heteroclinic tangency}\label{1_tangency}
\end{figure}


We look at the simplest bifurcation of the type described above.
Consider the arc $f_t$, $t\in[0,1]$, that starts with a linear
hyperbolic automorphism $L=f_0$ and ends with diffeomorphism $f=f_1$
with tangency at a heteroclinic point $R$. The difference is that
instead of creating a cubic tangency we create a transverse
quadratic tangency. The price we pay is that $f_1$ is only
$C^{1+Lip}$ smooth. Higher order derivatives do not exist at point
$L^{-1}(R)$.

\begin{theorem}
\label{main} Suppose that $L$ and $f$ are as in the previous
paragraph. The conjugacy between $L$ and $f$ and its inverse are
H\"older continuous with exponents equal to $1/2-\delta$ and
$1/4-\delta$. Number $\delta$ can be made arbitrarily small by an
appropriate choice of $L$ and $f$.
\end{theorem}

\begin{remark}
Number $1/4$ is not a sharp bound for the exponent. It clearly can
be improved. For a diffeomorphism with a heteroclinic tangency of
order $1+\alpha$, $0<\alpha<2$ our arguments imply that the
conjugacy and its inverse are H\"older continuous with exponents
$\frac1{1+\alpha}-\delta$ and $\frac{2-\alpha}{2(1+\alpha)}-\delta$.
Clearly such a diffeomorphism is only $C^{1+\alpha-\varepsilon}$. We
stick to the case $\alpha=1$ mainly to avoid cumbersome notation.
Notice that if $\alpha$ is close to zero then both exponents are
close to $1$. In the smoothness class $C^3$ and higher our arguments
fail.
\end{remark}

In the next section we point out that this construction also
provides an example of a system for which Mather spectrum differs
from the periodic one.

We heavily rely on the results in~\cite{E} as well as~\cite{BDV}
and~\cite{C}. Thus in the Section 3 we formulate results that are
relevant to our goal. In Section 4 we prove that the conjugacy is
H\"older continuous.

Finally in the last section we present a very short proof of a
positive result from the 2006 Ph.D. thesis of Travis Fisher that
complements ours.

 \begin{theorem}[\cite{F}]
 \label{fisher}
 A $C^{1+Lip}$ diffeomorphism that is conjugate to an Anosov one via a
H\"older conjugacy $h$ is Anosov itself provided that the product of
H\"older exponents for $h$ and $h^{-1}$ is greater than $1/2$.
 \end{theorem}

\begin{remark}
This result holds for any hyperbolic set as well. The proof is the
same. Also we remark that we have removed an unnecessary condition
that was present in the formulation of the result in~\cite{F}.
\end{remark}

\medskip

{\bfseries Acknowledgements.} Anatole Katok suggested the author to
look at the system with a heteroclinic tangency since this is the
simplest situation when one can hope to get a counterexample. The
author would like to thank A.~Katok for discussions, encouragement
and useful comments on the text itself. He also would like to thank
M.~Guysinsky for explaining the idea of the proof of
Theorem~\ref{fisher}.

\section{Periodic spectrum versus Mather spectrum}

Recall the definition of Mather spectrum.
Denote by $\Gamma(TM)$ the set of continuous vector fields with supremum norm. Given a diffeomorphism $f\colon M\to M$ define $f_*\colon\Gamma(TM)\to\Gamma(TM)$
$$
f_*v(\cdot)=Df\left(v(f^{-1}(\cdot))\right).
$$
The specrum $Q_f$ of the complexification of $f_*$ is called Mather spectrum of $f$.

\begin{theorem}[\cite{Math}]
\label{mather}
If non-periodic points of $f$ are dense then any connected component of $Q_f$ is an annulus centered at $0$. Diffeomorphism $f$ is Anosov if and only if $1\notin Q_f$.
\end{theorem}

Define periodic spectrum of a diffeomorphism. Given a periodic point
$x$ of period $p$ denote by $\{\lambda_1(x)^p,\ldots
\lambda_d(x)^p\}$ the set of absolute values of eigenvalues of
$Df^p(x)$. Then
$$
P_f\stackrel{\mathrm{def}}{=}\overline{\bigcup_{x \in
\textup{Per}(f)}\{\lambda_1(x),\ldots ,\lambda_d(x)\}}.
$$

The following is easy to prove.
\begin{prop}
Let $f$ be an Anosov diffeomorphism of $\mathbb T^2$. Then $P_f=Q_f\cap\mathbb R_+$.
\end{prop}

In contrast to above Theorems~\ref{main} and~\ref{mather} imply.
\begin{cor}
Diffeomorphism $f$ from Theorem~\ref{main} provide an example of a
diffeomorphism with dense set of periodic points such that $P_f\neq
Q_f\cap\mathbb R_+$.
\end{cor}

\section{First heteroclinic tangency at the boundary of Anosov systems}

Here we describe some results of~\cite{E},~\cite{BDV} and~\cite{C} that we need.

Let $L$ be hyperbolic automorphism of $\mathbb T^2$. Denote by $e_u$
and $e_s$ the eigenvectors of $L$ and by $\lambda>1$ the unstable
eigenvalue, $Le_u=\lambda e_u$. Let $P$ and $Q$ be two different
fixed points of $L$ and $R$ an intersection of the stable manifold
of $P$ and unstable manifold of $Q$. We may assume that distances to
$R$ from $P$ and $Q$ are equal. Also we assume that the size of a
ball containing $\{P,Q,R\}$ is much smaller than the size of
$\mathbb T^2$. Let $(x,y)$ be coordinates in the neighborhood of $R$
that make stable foliation horizontal and unstable foliation
vertical. Let $B$ be a small ball of radius $r$ centered at $R$.

Define $f_t=\theta_t\circ L$, $t\in[0,1]$ where $\theta_t\colon\mathbb T^2\to\mathbb T^2$ is identity outside of $B$ and given by the following formula on $B$
$$
\theta_t(x,y)=
\left(
  \begin{array}{cc}
    \;\;\;\cos(t\gamma(\rho)) & \sin(t\gamma(\rho)) \\
    -\sin(t\gamma(\rho)) & \cos(t\gamma(\rho)) \\
  \end{array}
\right)
\left(
  \begin{array}{c}
    x \\
    y \\
  \end{array}
\right)
$$
where $\rho=\sqrt{x^2+y^2}$ and $\gamma\colon[0,\infty)\to[0,\pi/2]$ is a $C^\infty$ map satisfying $\gamma(0)=\pi/2$; $\gamma(\rho)=0$ for $\rho\ge r$; $\gamma$ is strictly decreasing on $[0,r]$. Thus on every circle centered at $R$ $\theta_t$ is a rotation by an angle no greater than $\pi/2$. Value $\pi/2$ is achieved for $\rho=0$ and $t=1$. Let $v\in T_R\mathbb T^2$ be the unit vertical vector. From the definition of $f=f_1$ we have
$$
\lim_{n\to\pm\infty}(Df^n)v=0.
$$
Hence $f$ is not Anosov.

Denote by $\mathcal U$ and $\mathcal V$ neighborhoods of segments
$PR$ and $Qf^{-1}(R)$ that contain $B$ and $f^{-1}(B)$ respectively.

\begin{theorem}[\cite{E}, \cite{BDV}, \cite{C}]
There exist $r$ small enough and function $\gamma$
such that corresponding arc $f_t$ as defined above satisfies the following.
\begin{enumerate}
\item Diffeomorphisms $f_t$ are Anosov for $t<1$.
\item Diffeomorphism $f$ possesses invariant contracting and expanding foliations $W^s$ and $W^u$. The leaves of $W^s$ and $W^u$ are $C^1$ immersed curves. Denote by $E^s$ and $E^u$ distribution tangent to these foliations.
\item \label{cubic} Foliations $W^s$ and $W^u$ are transverse everywhere but along the orbit of $R$. At the point $R$ they have cubic tangency. Namely, there is $\tau>1$ such that for $S(x,y)\in B$
\begin{equation}
\tan\measuredangle(E^u(S),e_s)\ge\tau(x^2+y^2).
\end{equation}
Analogous inequality holds for distribution $E^s$.
\item \label{Eu_horizontal} For any $S\notin\mathcal U$
$$
\tan\measuredangle(E^u(S),e_u)<\varepsilon.
$$
Number $\varepsilon$ can be made arbitrarily small by the choice of
$L$, $f$ and $\mathcal U$. If we let
\begin{multline*}
\qquad\qquad
B_\infty=\{x\in\mathbb T^2:\;\exists i>0\;\;\;\mbox{such that}\\
 f^{-i}(x)\in B,\;\{f^{-i}(x), f^{-i+1}(x),\ldots x\}\subset\mathcal U\},
\end{multline*}
 then for any $S\notin B_\infty$
$$
\tan\measuredangle(E^u(S),e_u)<1.
$$
Analogous statement holds for $E^s$ and $\mathcal V$.
\item \label{top_conj} Diffeomorphism $f$ is conjugate to $L$ by a homeomorphism $h$, $h\circ f=L\circ h$.
\end{enumerate}
\end{theorem}

\begin{remark}
Technical statements~(T\ref{cubic}) and~(T\ref{Eu_horizontal}) are not stated explicitly in the papers quoted but they follow from the cone constructions that are carried out there.
\end{remark}

It may seem that since the size of $B$ is small $E^u$ is almost
vertical inside of $B_\infty$ and almost horizontal outside. In
fact, the transition through the boundary of $B_\infty$ is
continuous. Parameter $\tau$ increases when $r$ goes to zero.

We will be working with exactly the same construction, but $\theta$
must be chosen differently. Function $\gamma$ can be chosen
differently with $\gamma^{(r)}(0)=0$ for $r<2$ and $\gamma''(0)<0$.
Then $\theta\in C^{1+Lip}$ and the tangency is quadratic. This way
instead of~(T\ref{cubic}) we have
\begin{equation}
\label{quadratic}
 \tan\measuredangle(E^u(S),e_s)\ge\tau\sqrt{x^2+y^2},\;\; \tau>1.
\end{equation}

We outline proofs of~(T\ref{Eu_horizontal}) and~(\ref{quadratic}) at
the end of this section.

\begin{remark}
For conservative systems the theorem above was established
in~\cite{C}. Original proof~\cite{E} required that product of the
eigenvalues at $P$ is greater than $1$ while the product of the
eigenvalues at $Q$ is less than $1$. Assumption on the eigenvalues
at $P$ and $Q$ in~\cite{BDV} is even more restrictive. The main
motivation of~\cite{C} was to extend the example to systems with
homoclinic tangency. Our proof works for homoclinic intersection as
well. We work with a heteroclinic intersection only for convenience.
Also we would like to remark that our proof of H\"older continuity
works for the original construction in~\cite{E}. One needs to start
the isotopy with a $C^1$ small perturbation of $L$ that satisfies
above assumptions on eigenvalues at $P$ and $Q$ instead of starting
with $L$.
\end{remark}

Let us recall the proof of~(T\ref{top_conj}) from~\cite{E} since
this is the statement that we strengthen.

\begin{proof}
The main tool here is the following result of P.~Walters.
\begin{theorem}[\cite{W}]
Let $g\colon M\to M$ be an Anosov diffeomorphism. Then there exists an $\varepsilon_0>0$ such that for each $0<\varepsilon<\varepsilon_0$ there exists $\delta>0$ such that if $\tilde g\colon M\to M$ is a homeomorphism and if $d(g,\tilde g)<\delta$, then there exists a unique continuous map $h$ of $M$ onto $M$ with $h\circ \tilde g=g\circ h$ and $d(h,id)<\varepsilon$.
\end{theorem}
Apply the theorem for $g=L$ and $\tilde g=f$ to get semiconjugacy $h$ with
$d_{C^0}(h,id)<\varepsilon$. Note that $d_{C^0}(L,f)\to 0$ as $r\to 0$.
We have to take $r$ small enough so that $\varepsilon$ is smaller than constant associated
to the local product structure of $W^s$ and $W^u$. This guarantees that $h$ is injective.
Indeed if $h$ glues together some points on, say, unstable manifold then iterating forward
we get that $h$ glues together some points fixed distance apart. This is impossible since $h$ is close to identity.

Hence, by the invariance of domain theorem, $h$ is a homeomorphism.
\end{proof}

\begin{proof}[Sketch of proof of~(T\ref{Eu_horizontal})]
Cone constructions in~\cite{E} and~\cite{C} imply that $\forall
S\notin B_\infty$
$$
\tan\measuredangle(E^u(S), e_u)<1.
$$
Then given a small number $\varepsilon$ there exists $N$ such that
$\forall S\notin\cup_{i=0}^Nf^i(B_\infty)$
$$
\tan\measuredangle(E^u(S), e_u)<\varepsilon.
$$
It is possible to fatten $\mathcal U$ so that set $\tilde B_\infty$
that corresponds to new fattened neighborhood $\tilde{\mathcal U}$
of $PR$ contains $\cup_{i=0}^Nf^i(B_\infty)$ as shown on the
Figure~\ref{2_fatten_U}.


\begin{figure}[htbp]
\begin{center}

\begin{picture}(0,0)%
\includegraphics{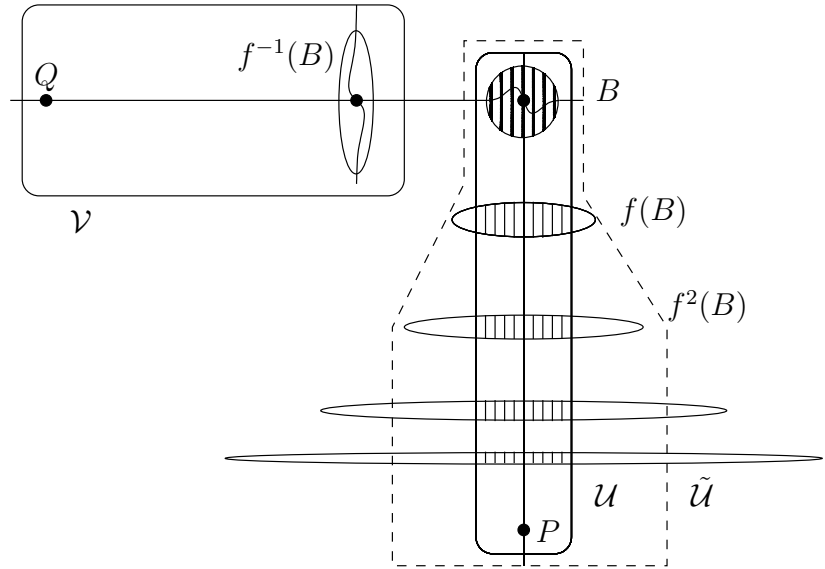}%
\end{picture}%
\setlength{\unitlength}{3947sp}%
\begingroup\makeatletter\ifx\SetFigFont\undefined%
\gdef\SetFigFont#1#2#3#4#5{%
  \reset@font\fontsize{#1}{#2pt}%
  \fontfamily{#3}\fontseries{#4}\fontshape{#5}%
  \selectfont}%
\fi\endgroup%
\begin{picture}(5120,3549)(364,-3298)
\put(751,-1186){\makebox(0,0)[lb]{\smash{{\SetFigFont{12}{14.4}{\rmdefault}{\mddefault}{\updefault}{\color[rgb]{0,0,0}$\mathcal V$}%
}}}}
\put(3676,-3136){\makebox(0,0)[lb]{\smash{{\SetFigFont{12}{14.4}{\rmdefault}{\mddefault}{\updefault}{\color[rgb]{0,0,0}$P$}%
}}}}
\put(4051,-361){\makebox(0,0)[lb]{\smash{{\SetFigFont{12}{14.4}{\rmdefault}{\mddefault}{\updefault}{\color[rgb]{0,0,0}$B$}%
}}}}
\put(4201,-1111){\makebox(0,0)[lb]{\smash{{\SetFigFont{12}{14.4}{\rmdefault}{\mddefault}{\updefault}{\color[rgb]{0,0,0}$f(B)$}%
}}}}
\put(4501,-1711){\makebox(0,0)[lb]{\smash{{\SetFigFont{12}{14.4}{\rmdefault}{\mddefault}{\updefault}{\color[rgb]{0,0,0}$f^2(B)$}%
}}}}
\put(4051,-2911){\makebox(0,0)[lb]{\smash{{\SetFigFont{12}{14.4}{\rmdefault}{\mddefault}{\updefault}{\color[rgb]{0,0,0}$\mathcal U$}%
}}}}
\put(4651,-2911){\makebox(0,0)[lb]{\smash{{\SetFigFont{12}{14.4}{\rmdefault}{\mddefault}{\updefault}{\color[rgb]{0,0,0}$\tilde{\mathcal U}$}%
}}}}
\put(526,-265){\makebox(0,0)[lb]{\smash{{\SetFigFont{12}{14.4}{\rmdefault}{\mddefault}{\updefault}{\color[rgb]{0,0,0}$Q$}%
}}}}
\put(1801,-136){\makebox(0,0)[lb]{\smash{{\SetFigFont{12}{14.4}{\rmdefault}{\mddefault}{\updefault}{\color[rgb]{0,0,0}$f^{-1}(B)$}%
}}}}
\end{picture}%

\end{center}
\caption{The hatched set is $B_{\infty}$. Distance $|QR|$ is much
bigger than $f^N(B)=f^2(B)$. Hence it is possible to fatten
$\mathcal U$ to $\tilde{\mathcal U}$ so that unstable distribution
outside $\tilde B_\infty$ is $\varepsilon$-close to horizontal
vector $e_u$.}\label{2_fatten_U}
\end{figure}


For that we need to make sure that $f^N(B)$ is small compared to the
distance $|QR|$. This can be achieved by appropriate choice of
automorphism $L$, $P$ and $Q$.

We fix a hyperbolic matrix $L$ that induces an automorphism of
$\mathbb T^2=\mathbb R^2/\mathbb Z^2$ with $d$ fixed points. We fix
size $r$ of the ball $B$ and the map $\theta|_B$. The trick now is
to choose the torus $\mathbb T^2$ to be ``big" when compared to
eigenvalue $\lambda$ of $L$ and $r$.

Linear map $L$ induces a hyperbolic automorphism of $\mathbb
T^2=\mathbb R^2/k\mathbb Z^2$, where $k$ is a big integer. This
automorphism is a finite cover of the automorphism of $\mathbb
T^2=\mathbb R^2/\mathbb Z^2$. It also has $d$ fixed points.
Obviously the distances between those fixed points are big now.
Hence $P$ and $Q$ can be chosen so that $|QR|$ is big.
\end{proof}
\begin{remark}
We will use the fact that $|QR|$ can be chosen big independently of
$r$ and $\lambda$ several times in the course of the proof of
H\"older continuity.
\end{remark}

\begin{prop}
\label{propos} Given a point $S\in\mathcal U$. Denote by $d(S)$ the
distance to $PR$. Then
\begin{equation}
\label{cubic_for E^s} \tan\measuredangle(E^s(S),e_s)\le
\kappa\,d(S)^2.
\end{equation}
where $\kappa$ is a number that depends on $\varepsilon$
from~(T\ref{Eu_horizontal}).
\end{prop}

\begin{proof}[Proof of Proposition]
Let $N$ be the smallest positive integer such that
$f^N(S)\notin\mathcal U$. Then $\lambda^nd(S)\approx1$ and
$\tan\measuredangle(E^s(f^N(S)),e_s)\le\varepsilon$
by~(T\ref{Eu_horizontal}) since $f^N(S)\notin\mathcal V$ as well.
$$
\tan\measuredangle(E^s(S),e_s)=\lambda^{-2n}\tan\measuredangle(E^s(f^N(S)),e_s)\le
C\varepsilon d(S)^2.
$$
\end{proof}

\begin{proof}[Sketch of proof of~(\ref{quadratic})]
Denote by $A$ and $\tilde A$ points of intersection of the line $QR$
and boundary $\partial B$ as shown on Figure~\ref{3_q_variation}. It
follows follows from the definition of $f$ that straight segment
$[Q,A]$ is inside of $W^u(Q)$ while straight segment $[P,R]$ is
inside $W^s(P)$. The shape of $W^u(Q)$ between $A$ and $\tilde A$ is
completely determined by $\theta$. Namely, it is the image of the
segment $[A, \tilde A]$ under $\theta$. We remark that on the other
hand the shape of $W^u(Q)$ between $A$ and $R$ determines $\theta$
since map $\theta$ is a rigid rotation on every circle around $R$.


\begin{figure}[htbp]
\begin{center}

\begin{picture}(0,0)%
\includegraphics{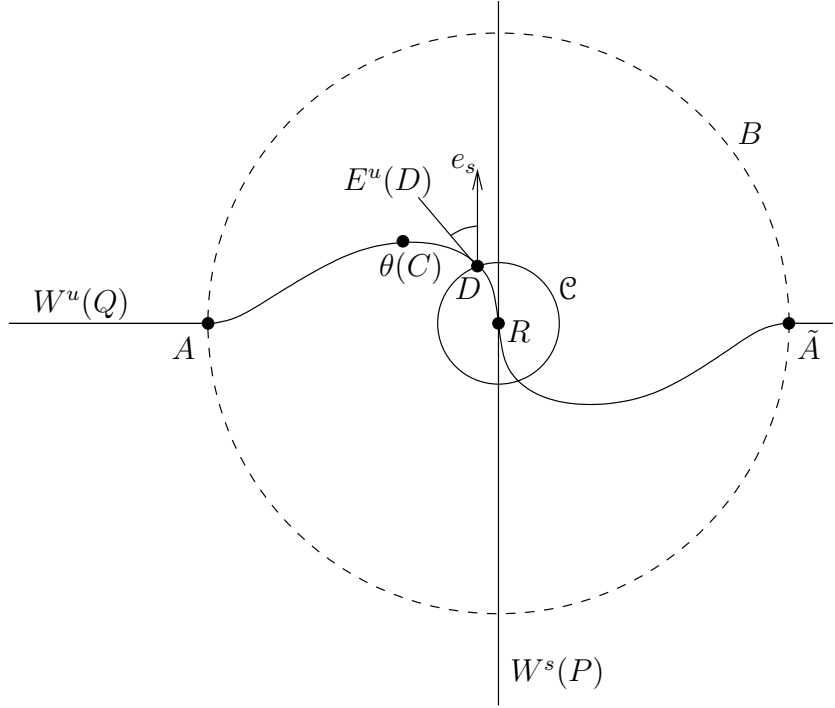}%
\end{picture}%
\setlength{\unitlength}{3947sp}%
\begingroup\makeatletter\ifx\SetFigFont\undefined%
\gdef\SetFigFont#1#2#3#4#5{%
  \reset@font\fontsize{#1}{#2pt}%
  \fontfamily{#3}\fontseries{#4}\fontshape{#5}%
  \selectfont}%
\fi\endgroup%
\begin{picture}(5199,4449)(-1286,-3373)
\put(1525,-805){\makebox(0,0)[lb]{\smash{{\SetFigFont{12}{14.4}{\rmdefault}{\mddefault}{\updefault}{\color[rgb]{0,0,0}$D$}%
}}}}
\put(3301,164){\makebox(0,0)[lb]{\smash{{\SetFigFont{12}{14.4}{\rmdefault}{\mddefault}{\updefault}{\color[rgb]{0,0,0}$B$}%
}}}}
\put(3676,-1186){\makebox(0,0)[lb]{\smash{{\SetFigFont{12}{14.4}{\rmdefault}{\mddefault}{\updefault}{\color[rgb]{0,0,0}$\tilde A$}%
}}}}
\put(1051,-661){\makebox(0,0)[lb]{\smash{{\SetFigFont{12}{14.4}{\rmdefault}{\mddefault}{\updefault}{\color[rgb]{0,0,0}$\theta(C)$}%
}}}}
\put(2176,-811){\makebox(0,0)[lb]{\smash{{\SetFigFont{12}{14.4}{\rmdefault}{\mddefault}{\updefault}{\color[rgb]{0,0,0}$\EuScript C$}%
}}}}
\put(1501, 14){\makebox(0,0)[lb]{\smash{{\SetFigFont{12}{14.4}{\rmdefault}{\mddefault}{\updefault}{\color[rgb]{0,0,0}$e_s$}%
}}}}
\put(1876,-3211){\makebox(0,0)[lb]{\smash{{\SetFigFont{12}{14.4}{\rmdefault}{\mddefault}{\updefault}{\color[rgb]{0,0,0}$W^s(P)$}%
}}}}
\put(-1124,-886){\makebox(0,0)[lb]{\smash{{\SetFigFont{12}{14.4}{\rmdefault}{\mddefault}{\updefault}{\color[rgb]{0,0,0}$W^u(Q)$}%
}}}}
\put(-251,-1186){\makebox(0,0)[lb]{\smash{{\SetFigFont{12}{14.4}{\rmdefault}{\mddefault}{\updefault}{\color[rgb]{0,0,0}$A$}%
}}}}
\put(826,-121){\makebox(0,0)[lb]{\smash{{\SetFigFont{12}{14.4}{\rmdefault}{\mddefault}{\updefault}{\color[rgb]{0,0,0}$E^u(D)$}%
}}}}
\put(1849,-1081){\makebox(0,0)[lb]{\smash{{\SetFigFont{12}{14.4}{\rmdefault}{\mddefault}{\updefault}{\color[rgb]{0,0,0}$R$}%
}}}}
\end{picture}%

\end{center}
\caption{Establishing linear variation of the angle on the circle
$\EuScript C$.}\label{3_q_variation}
\end{figure}


First let us establish linear variation of the
angle~(\ref{quadratic}) along $\theta[A,R]$. It follows from the
choice of $\theta$. Let $C\in[A,R]$ be the point such that
$\measuredangle(E^u(\theta(C)),e_s)=\pi/2$. Then for any
$S\in\theta[C,R]$ we have
$$
\tan\measuredangle(E^u(S),e_s)\ge\tau\sqrt{x^2+y^2}
$$
since the tangency is quadratic. For $S\in\theta[A,C]$
$$
\tan\measuredangle(E^u(S),e_s)\ge c
$$
where $c$ is a big constant, $c\gg 1$.

Fix a circle $\EuScript C$ of radius $\rho$ centered at $R$. Let
$D=\EuScript C\cap\theta[A,R]$. We know that
estimate~(\ref{quadratic}) holds for $D$ and we would like to
establish it for other points on $\EuScript C$.

Consider distribution $\tilde E^u\subset T_{\EuScript C}\mathbb T^2$
given by the formula $\tilde E^u=DL\,E^u$. Then $E^u$ on $\EuScript
C$ is given by $E^u=Df E^u=D\theta\tilde E^u$.

If we denote by $v$ and $u$ normal and tangent vector fields to
$\EuScript C$ then $D\theta$ with respect to bases $(v(\cdot),
u(\cdot))$ and $(v(\theta(\cdot)),u(\theta(\cdot)))$ is given by the
shear matrix
\begin{equation}
\label{d_teta}
 D\theta=\left(
  \begin{array}{cc}
    1 & \alpha \\
    0 & 1 \\
  \end{array}
\right)\!,\; \alpha=\alpha(\rho)>0.
\end{equation}

According to Proposition~\ref{propos}
$$
\tan\measuredangle(\tilde E^u(S),e_u)\le\kappa\rho^2,\;\; S\in
\EuScript C.
$$
Meanwhile
$$
\tan\measuredangle(E^u(D),e_s)\ge\tau\rho.
$$
These inequalities together with the formula for $D\theta$ above
imply linear variation of angle~(\ref{quadratic}) for all
$S\in\EuScript C$.
\end{proof}

\section{Topological conjugacy is H\"older}

Here we prove that the conjugacy $h$ and its inverse are H\"older continuous.

We mimic the standard proof of H\"older continuity of conjugacy
between two Anosov systems (e.~g. Section 19.1 in~\cite{KH}). The
conjugacy maps stable and unstable foliations of $f$, $W^s$ and
$W^u$, into stable and unstable foliations of $L$, $W_L^s$ and
$W_L^u$, respectively.
\begin{step}
A restriction of $h$ to a leaf of $W^u$ is H\"older continuous with the exponent
equal to $1-\delta$. A restriction of $h^{-1}$ to a leaf of $W_L^u$ is H\"older continuous with the exponent equal to $1/4-\delta$. Number $\delta$ depends on the choice of $L$ and $r$ and can be made arbitrarily small.
 \end{step}
\begin{remark}
The distances with respect to which we show H\"older continuity are
induced Riemannian distances
 on the leaves of $W^u$ and $W_L^u$.
\end{remark}

Analogously, $h$ and $h^{-1}$ are H\"older continuous when
restricted to stable leaves with exponents $1-\delta$ and
$1/4-\delta$ respectively. We immediately get the following.
\begin{prop}
Homeomorphism $h^{-1}$ is H\"older continuous with exponent
$1/4-\delta$.
\end{prop}
This follows from a standard argument that utilizes uniform
transversality of $W_L^s$ and $W_L^u$.

To conclude that $h$ is H\"older as well one needs to have $W^s$ and $W^u$ to be uniformly transversal. In our situation $W^s$ and $W^u$ are uniformly transversal only outside a neighborhood of the orbit of $R$. This leads to further loss of the exponent by factor of $1/2$.
\begin{step}
Conjugacy $h$ is H\"older continuous with exponent equal to $1/2-\delta$.
\end{step}
Heuristically it is clear that the loss of the exponent at the second step is inevitable. In the second step we ``straighten out" the quadratic tangency. Thus the exponent is no greater than $1/2$.

Together with~(\ref{quadratic}) Proposition~\ref{propos} gives us
control on the angle between $E^s$ and $E^u$ in the neighborhood of
the orbit of $R$ which is crucial for carrying out estimates in Step
2.

Throughout the proof we denote by $d^u$, $d^s$, $d_L^u$ and $d_L^s$
induced Riemannian distances along the leaves of corresponding
foliations.

\begin{proof}[Proof of Step 1]
Uniform continuity of $h$ implies that $\exists C_1>0$ such that
\begin{multline*}
\frac{1}{C_1}\, d_L^u(h(a),h(b))\le d^u(a,b)\le C_1d_L^u(h(a),h(b))\;\;\;\;\;
\mbox{whenever}\;\;\;\;\; d^u(a,b)\ge\frac {r}{10}
\end{multline*}
To prove H\"older estimates for close-by points $a$ and $b\in
W^u(a)$ we need to have exponential estimates on expansion along
$E^u$. Given a point $a\in\mathbb T^2$ let
$$
D^u(a)=\|Df(a)v^u\|,\;\;\;v^u\in E^u, \|v^u\|=1.
$$
If $a\notin L^{-1}(B)$ then, obviously, $\lambda^{-1}\le
D^ua\le\lambda$. If $a\in L^{-1}(B)$ then $D^ua$ can be bigger.
Still there exists $D$ such that $D^ua\le D$ for $a\in L^{-1}(B)$.
In fact, using estimates on $\alpha$ from~(\ref{d_teta}) one can
show that $D=2\lambda$ works but we will not use it. Keeping $r$ and
$\theta$ fixed choose $L$ (e.~g. pass to a finite cover as in proof
of~(T\ref{Eu_horizontal})), $P$ and $Q$ so that $|PR|$ is large and
hence first return time $m$ to $L^{-1}(B)$ is big. Thus $\forall
a\in\mathbb T^2$
\begin{equation}
\label{product}
\prod_{i=1}^{m-1}D^uf^i(a)\le\lambda^{m-1}D\le(\lambda^{1+\delta})^m,
\end{equation}
where $\delta=\delta(m, D)$ is a small number. It follows that
\begin{equation}
\exists C_2 :\;\; \forall n>0\; \|Df^nv^u\|\le C_2 (\lambda^{1+\delta})^n\|v^u\|,\;v^u\in E^u.
\end{equation}

Now we use standard argument to prove H\"older continuity. Take
$a\in W^u(b)$ close to $b$. Let $N$ be the smallest number such that
$d^u(f^N(a),f^N(b))\ge r/10$. Then
\begin{multline*}
d_L^u(h(a),h(b))=\frac{1}{\lambda^N}d_L^u\left(L^N(h(a)),L^N(h(b))\right)\le\frac{C_1}{\lambda^N}d^u(f^N(a),f^N(b))\\
\le\frac{C_1C_3}{\lambda^N}d^u(f^N(a),f^N(b))^{1-\delta}\le\frac{C_1C_2^{1-\delta}C_3}{\lambda^N}(\lambda^{1+\delta})^{N(1-\delta)}d^u(a,b)^{1-\delta}\le Cd^u(a,b)^{1-\delta}.
\end{multline*}

To show that $h^{-1}$ is H\"older along $W_L^u$ we need an estimate
on the product in~(\ref{product}) from below. According
to~(T\ref{Eu_horizontal})  if $a\notin \mathcal U\cup L^{-1}(B)$
then $D^u(a)\ge\mu$, where $\mu=\mu(\varepsilon)$ and
$\mu\nearrow\lambda$ when $\varepsilon\to 0$ (we remark that the
choices we do to make $\varepsilon$ small do not affect $\lambda$).
If $a\in \mathcal U\backslash B_\infty$ then another inequality
from~(T\ref{Eu_horizontal}) provide an estimate on expansion.

If $a\in B_{\infty}$ then, obviously, $D^u(a)\ge\lambda^{-1}$ but we
need to have a better control. Split $B_{\infty}$ into its connected
components
$$B_{\infty}=\bigcup_{i\ge 0}B_i$$ (see Figure~\ref{2_fatten_U}). Let $(x^i,y^i)$ be coordinates in $B_i$ obtained by parallel
transport of $(x,y)$ from $R$ to $f^i(R)$. Consider rectangles
\begin{equation}
\label{bar1}
\bar B_i=\{(x^i,y^i) : |x^i|\le r\lambda^{-i}, |y^i|\le r\lambda^{-3i} \}.
\end{equation}
Also let
\begin{equation}
\label{bar2}
\bar B=\bigcup_{i\ge 0}\bar B_i.
\end{equation}

Let $\xi=\tau r$ where $\tau$ is from~(\ref{quadratic}).
\begin{lemma}
\label{lemma}
Consider a point $S(x^n,y^n)\in B_n$,  $S\notin \bar B_n$. Then
$$
\tan\measuredangle(E^u(S),e_s)\ge\xi.
$$
\end{lemma}
\begin{proof}
Note that this follows from~(\ref{quadratic}) if $n=0$.

$$
\tan\measuredangle(E^u(S),e_s)=\lambda^{2n} \tan\measuredangle(E^u(f^{-n}(S)),e_s).
$$
Clearly $(\lambda^{-n}x^n,\lambda^ny^n)$ are $(x,y)$ coordinates of $f^{-n}(S)$. If $|x^n|\ge r\lambda^{-n}$ then
  $$
  \lambda^{2n} \tan\measuredangle(E^u(f^{-n}(S)),e_s)\ge \lambda^{2n}\tau(\lambda^{-n}x^n)\ge\tau r.
  $$
If $|y^n|\ge r\lambda^{-3n}$ then
   $$
   \lambda^{2n} \tan\measuredangle(E^u(f^{-n}(S)),e_s)\ge
   \lambda^{2n}\tau(\lambda^{n}y^n)\ge\tau r.
   $$                                                           \end{proof}

Take an arbitrary point $a\notin\bar B$. Then
$\tan\measuredangle(E^u(a),e^s)\ge\xi$ and hence there exists $m>0$
such that $D^uf^m(a)>\mu$. Therefore contraction along unstable
mostly happens only inside $\bar B$. After the point leaves $\bar B$
its unstable direction ``recovers" after $m$ steps even before the
point leaves $\mathcal U$.

Now we will be doing estimates from below on expansion along $E^u$
by analyzing itinerary of a point.

As before (sketch of proof of (T\ref{Eu_horizontal})) we make sure
by the choice of $L$ that after a point leaves $\mathcal U$ it
spends at least time $4m+4$ before it enters $\mathcal V$.

Take a point $a(x,y)\in B$ and  assume that $f^i(a)\in \bar B_i$,
$i=0,\ldots n$, $f^{n+1}(a)\notin \bar B_{n+1}$. By the Lemma
$|x|\le r\lambda^{-2n}$, $|y|\le r\lambda^{-2n}$. Hence
$f^{-1}(a),\ldots f^{-2n}(a)\in\mathcal V$, $D^uf^{-i}(a)\ge\mu$,
$i=2,\dots 2n$. Simple calculation with $D\theta$ shows that
$D^uf^{-1}(a)\ge 1$. Start with $f^{-2n}(a)$ and wait time
$4n+4m+4$. We know for sure that during that time the orbit has not
entered $\mathcal V$ again. Thus
\begin{multline}
\label{period_estimate}
\prod_{i=-2n}^{2n+4m+4}D^uf^i(a)=\prod_{i=-2n}^{-1}D^uf^i(a)
\prod_{i=0}^{n+m}D^uf^i(a)
\prod_{i=n+m+1}^{2n+4m+4}D^uf^i(a)\\
\ge\mu^{2n-1}\lambda^{-n-m-1}\mu^{n+3m+4}=\lambda^{-n-m-1}\mu^{3n+3m+3}.
\end{multline}
This estimate is good enough to get exponent $1/2-\delta$. The only
problem is that it holds only along specific orbit segments
described above. Call them ``cycles".

As before take $a\in W^u(b)$ close to $b$ and let $N$ be the smallest number such that $d^u(f^N(a),f^N(b))\ge r/10$.

Let us first explain the idea informally. We have to study how the
length of $d^u(f^i(a),f^i(b))$ changes as $i=0,\ldots N$. We
decompose this time segment into ``cycles" as
in~(\ref{period_estimate}). These ``cycles" do not overlap. There
might be some ``gaps" between the ``cycles" that only improve the
estimate since the time spent in the ``bad" set $\bar B_\infty$ is
inside the ``cycles". The difficulty that we have to deal with is
that at the beginning ``cycle" might be ``incomplete". The ``worst"
situation is when $a$ is close to $R$. Same problem occurs at the
end --- the last ``cycle" might happen to be ``cut" at the end.

Notice that in the description above we ignore returns to $\mathcal
U\backslash B$. The expansion rate inside this set might be less
than $\mu$ but still is greater than 1 according
to~(T\ref{Eu_horizontal}). Therefore these returns are much easier
to take care of. Hence we consider only ``cycles" that correspond to
returns to $B$.

The problem of the ``cut" at the end is easy to deal with. Denote by
$t_j$, $j=1,\ldots l$ the lengths of ``cycles". We count incomplete
``cycles" as well. We also consider numbers $n_j$, $j=1,\ldots l$.
In the notation of~(\ref{period_estimate}) $t_j=4n_j+4m+4$. Assume
that the last ``cycle" number $l$ is incomplete. During the last
``cycle" the segment enters $B$, spends time $n_l$ inside $\bar
B_{\infty}$ and then it recovers to the size of $r/10$ during the
time less than $n_l+4m+4$. When the segment leaves $\bar B_{n_l}$ it
has length of order $\lambda^{-3n_l}$ since that is the vertical
size of $\bar B_{n_l}$ and the unstable foliation is roughly
vertical in $\bar B_{n_l}$. Clearly number $n_l$ is not big since
otherwise the segment cannot recover to the size $r/10$. We can
control $N$ independently of $n_l$ by choosing $a$ and $b$ extremely
close to each other. This may result in big H\"older constant but
the exponent will not be affected. In fact, since $n_l$ is bounded
this difficulty at the end of ``time-window" $i=0,\ldots N$ can be
taken care of by the H\"older constant.

Now assume that the first ``cycle" is incomplete as well. In contrast to above $n_1$ might happen to be big when compared to $N$ since the size of the segment is small at the beginning. Clearly we need to examine the ``worst" situation when $a\in B$. As before we can argue that after time $n_1$ the size of the segment is of the order $\lambda^{-3n_1}$. Number $N-n_1$ is greater than $3n_1$ since during this time the segment grows up to size $r/10$.
Hence, using~(\ref{period_estimate}) we get
\begin{multline}
\label{prelim}
d^u(f^N(a), f^N(b))\approx \lambda^{-n_1}\lambda^{n_1}\lambda^{\frac34(N-2n_1)}\lambda^{-\frac14(N-2n_1)}d^u(a,b)\\
=\lambda^{-\frac12(N-2n_1)}d^u(a,b)=\lambda^{-\frac12(\frac N2+\frac N2-2n_1)}d^u(a,b)
\ge\lambda^{\frac N4}d^u(a,b).
\end{multline}
This clearly good enough to get exponent $\frac14-\delta$.

From now on we will be providing details to the scheme described above. Still we stay away from completely rigorous technical discussion. The technical details are plentiful while the way argument works is fairly transparent.

To get the estimate $N\ge 4n_1$ we need to redefine slightly the set $\bar B$, numbers $n_i$ and $t_i$ accordingly. First fix $\mu$ close to $\lambda$, $\mu<\lambda$.
\begin{equation*}
\bar B_i=\{(x^i,y^i) : |x^i|\le \tilde r\lambda^{-i}, |y^i|\le \tilde r\lambda^{-3i} \},
\end{equation*}
\begin{equation*}
\bar B=\bigcup_{i\ge 0}\bar B_i,
\end{equation*}
where $\tilde r=\tilde r(\mu)<r/20$ is chosen so that $D^u(x)\ge\mu^{-1}$ for any $x\in B\backslash \bar B_0$ and by the arguments of Lemma~\ref{lemma} $D^u(x)\ge\mu^{-1}$ for any $x\notin \bar B_\infty$.

We study lengths of the segment during the time $N$ defined above.
We consider the ``worst" case when $a$ and $b$ are close to $R$. Let
$n_1$ be the smallest integer such that
\begin{equation*}
\frac{d^u(f^{n_1+1}(a),f^{n_1+1}(b))}{d^u(f^{n_1}(a),f^{n_1}(b))}\ge\mu^{-1}.
\end{equation*}
Then it is easy to see that after some fixed number of iterates $m=m(\mu)$ which is independent of $n_1$ we will have
\begin{equation*}
\frac{d^u(f^{n_1+m+1}(a),f^{n_1+m+1}(b))}{d^u(f^{n_1+m}(a),f^{n_1+m}(b))}\ge\mu.
\end{equation*}

\begin{remark}
Notice that any iterate of the segment does not cross more than one
connected component of $\bar B_\infty$. Indeed, the vertical
distance between $\bar B_n$ and $\bar B_{n+1}$ is of order
$\lambda^{-n}$, horizontal sizes of $\bar B_n$ and $\bar B_{n+1}$
are of order $\lambda^{-3n}$ while in the ``gap" between $B_n$ and
$B_{n+1}$ unstable foliation is in horizontal cone field according
to~(T\ref{Eu_horizontal}). It follows that local unstable leaves do
not intersect $\bar B_n$ and $\bar B_{n+1}$.
\end{remark}

Analogously define numbers $n_j$, $j=1,\ldots l$. Then define $t_j=n_j+4m+4$ and corresponding ``cycles" as before. Clearly we have an analogue of~(\ref{period_estimate})
\begin{equation}
\label{period_estimate_new}
d^u(f^{s_j+t_j}(a),f^{s_j+t_j}(b))\ge\lambda^{-n-m-1}\mu^{3n+3m+3}d^u(f^{s_j}(a), f^{s_j}(b)).
\end{equation}
where $s_j$ is the starting time of a full ``cycle". If the
itinerary of the segment has complete first ``cycle" then the same
way as in the proof of H\"older continuity of $h$
using~(\ref{period_estimate_new}) we get
$$
d^u(h^{-1}(a),h^{-1}(b))\le C d^u(a,b)^{\frac12-\delta}
$$
with $\delta=\delta(\mu)\searrow 0$ as $\mu\nearrow\lambda$.

Now we go back to  the ``worst" case when the first ``cycle" starts
near $R$. Definitions of $\bar B_\infty$ and $n_1$ guarantee that at
least half of the segment $[f^{n_1-1}(a),f^{n_1-1}(b)]\subset
W^u(f^{n_1-1}(a))$ lies inside of $\bar B_{n_1-1}$. Recall that
unstable foliation is almost vertical inside $\bar B_\infty$. Hence
$$
d^u(f^{n_1-1}(a), f^{n_1-1}(b))\le \frac r{10}\lambda^{-3(n_1-1)}.
$$
Recall that $D^u(x)\le\lambda$ for any $x\notin L^{-1}(B)$ and $d^u(f^N(a),f^N(b))\ge r/10$. It follows that $N-n_1\ge\Gamma n_1$ where $\Gamma$ depends on frequency of visits to $L^{-1}(B)$ and can be made arbitrarily close to $3$. Now the preliminary estimate~(\ref{prelim}) transforms into the following one

\begin{multline*}
d^u(f^N(a),f^N(b)) \ge\lambda^{-n_1-m} \mu^{n_1+3m+4}\mu^{\frac34(N-2n_1)}
\lambda^{-\frac14(N-2n_1)}\\
\ge\lambda^{-n_1-m} \mu^{n_1+3m+4}(\mu/\lambda)^{\frac14(N-2n_1)}
\mu^{\frac12(N-2n_1)}\\
=\lambda^{-n_1-m}\mu^{n_1+3m+4}(\mu/\lambda)^{\frac14(N-2n_1)} \mu^{\frac12(1-\Delta)(N-2n_1)}\mu^{\frac12\Delta(N-2n_1)}\\
\ge\mu^{\frac12\Delta(N-2n_1)}\ge\mu^{\frac14(N-\delta)},
\end{multline*}
where $\delta=\delta(\Gamma,\Delta)$ is small and $\Delta=\Delta(\lambda,\mu)$ is chosen so that $\mu^{\frac12(1-\Delta)(N-2n_1)}$ compensates the factors in front of it compensates the factor in front of . Number $\Delta\nearrow 1$ as $\mu\nearrow\lambda$. It follows that $\delta$ can be arbitrarily small.
\end{proof}

\begin{proof}[Proof of Step 2]
First of all let us notice that outside of $\mathcal U$ foliations $W^s$ and $W^u$ are uniformly transversal. Then, by the standard argument, H\"older continuity along $W^s$ and $W^u$ implies H\"older continuity of $h$ with exponent $1-\delta$ outside of $\mathcal U$.

It follows from~(\ref{quadratic}) and Proposition~\ref{propos} that
inside $B$ the angle  $\measuredangle(E^s,E^u)$ varies linearly with
distance to $R$. This allows to show that $h$ is H\"older continuous
with exponent $1/2-\delta$ inside $B$. More work is required to
establish H\"older continuity in the rest of $\mathcal U$. We start
with an observation that allows to reduce our task to establishing
H\"older inequality for points inside of a single $B_n$, $n\ge 0$.

Introduce vertical and horizontal cones
$$
\mathcal C_v(x)=\{v\in T_x\mathbb T^2 : \tan\measuredangle(v,e_s)<\varepsilon\},
$$
$$
\mathcal C_h(x)=\{v\in T_x\mathbb T^2 : \tan\measuredangle(v,e_s)>\xi\}
$$
with $\varepsilon$ as in~(T\ref{Eu_horizontal}) and $\xi$ as in Step
1. These cones have disjoint interiors. Moreover, $E^s(x)\in\mathcal
C_v(x)$ for any $x\in\mathcal U$ by~(T\ref{Eu_horizontal}) and
$E^u(x)\in\mathcal C_h(x)$ for any $x\in \mathcal U\backslash\bar
B_\infty$ by Lemma~\ref{lemma}. Thus $\mathcal C_v$ and $\mathcal
C_h$ provide good control of $E^s$ and $E^u$ in $\mathcal
U\backslash\bar B_\infty$.

Take $a$ and $b$ closeby inside $\mathcal U$. Let $e$ be the
intersection of local unstable manifold $W^u(a,r/10)$ and local
stable manifold $W^s(b,r/10)$.

\begin{lemma}
\label{lemma2}
Assume that some fixed proportion with respect to the length of local unstable manifold connecting $a$ and $e$ lies inside of $\mathcal C_h$ --- meaning that tangent vector is in the cone. Then
$$
d(h(a),h(b))\le Cd(a,b)^{1-\delta}.
$$
\end{lemma}
Recall that we know that local stable manifold is in $\mathcal C_v$. Then the proof of the Lemma is a straightforward adjustment of the standard one when the whole local unstable manifold lies inside of $\mathcal C_h$ as well.

We have remarked in the  course of the proof of Step 1 that local
unstable leaves do not meet different connected components of $\bar
B_\infty$. Together with Lemma~\ref{lemma2} this implies that we are
only left to deal with points $a$ and $b$ such that local unstable
manifold connecting $a$ and $e$ lies almost entirely in $\bar B_n$
for some $n\ge 0$.

Denote by $(x,y)$ the coordinate system centered at $f^n(R)$ with $x$-axis being horizontal. Observe further that it is enough to consider points $a$ and $b$ that have the same $y$-coordinate. Let $a=a(x_1,y_1)$, $b=b(x_2,y_1)$ and $e=e(x_3,y_3)$. We are aiming at proving the estimate
\begin{equation}
\label{beak}
C|x_1-x_2|^{\frac12}\ge|y_1-y_3|.
\end{equation}
Together with Step 1 this would imply that $h$ is H\"older with exponent $\frac12-\delta$.

Given a point $S(x,y)\in B_n$ we have
\begin{multline*}
\tan\measuredangle(E^u(S),e_s)=\lambda^{2n}\measuredangle(E^u(f^{-n}(S)),e_s)\ge
\lambda^{2n}\tau\sqrt{(\lambda^{-n}x)^2+(\lambda^ny)^2}\\
\ge\lambda^{2n}\tau\lambda^{-n}\sqrt{x^2+y^2}\ge\sqrt{x^2+y^2}
\end{multline*}
and by Proposition~\ref{propos}
$$
\tan\measuredangle(E^s(S),e_s)\le\kappa d(S)^2\le\frac{x^2}2.
$$
These inequalities provide control on the angle needed to carry out the estimate~(\ref{beak}).

Denote by $\EuScript B$ the ``beak" formed by the unstable manifold
connecting $a$ and $e$ and the stable manifold connecting $b$ and
$e$. We will consider two representative cases illustrated on
Figure~\ref{4_beaks}.


\begin{figure}[htbp]
\begin{center}

\begin{picture}(0,0)%
\includegraphics{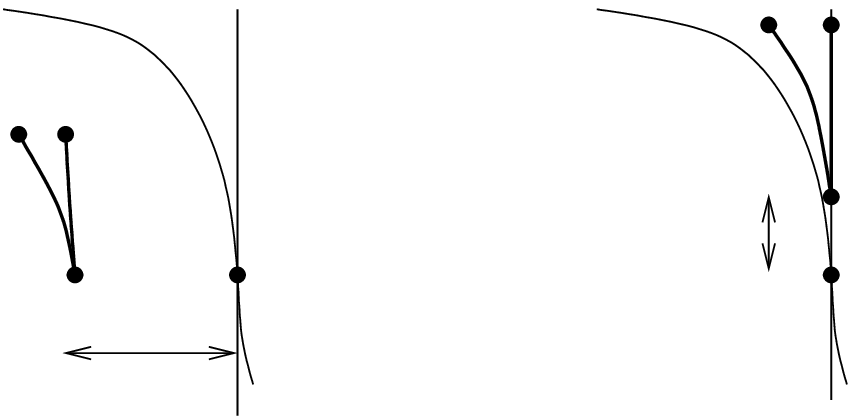}%
\end{picture}%
\setlength{\unitlength}{3947sp}%
\begingroup\makeatletter\ifx\SetFigFont\undefined%
\gdef\SetFigFont#1#2#3#4#5{%
  \reset@font\fontsize{#1}{#2pt}%
  \fontfamily{#3}\fontseries{#4}\fontshape{#5}%
  \selectfont}%
\fi\endgroup%
\begin{picture}(4080,2052)(61,-1198)
\put(577,683){\makebox(0,0)[lb]{\smash{{\SetFigFont{12}{14.4}{\rmdefault}{\mddefault}{\updefault}{\color[rgb]{0,0,0}$W^u(Q)$}%
}}}}
\put(1276,-586){\makebox(0,0)[lb]{\smash{{\SetFigFont{12}{14.4}{\rmdefault}{\mddefault}{\updefault}{\color[rgb]{0,0,0}$f^n(R)$}%
}}}}
\put(4126,-586){\makebox(0,0)[lb]{\smash{{\SetFigFont{12}{14.4}{\rmdefault}{\mddefault}{\updefault}{\color[rgb]{0,0,0}$f^n(R)$}%
}}}}
\put(1276,539){\makebox(0,0)[lb]{\smash{{\SetFigFont{12}{14.4}{\rmdefault}{\mddefault}{\updefault}{\color[rgb]{0,0,0}$W^s(P)$}%
}}}}
\put( 76,239){\makebox(0,0)[lb]{\smash{{\SetFigFont{12}{14.4}{\rmdefault}{\mddefault}{\updefault}{\color[rgb]{0,0,0}$a$}%
}}}}
\put(376,239){\makebox(0,0)[lb]{\smash{{\SetFigFont{12}{14.4}{\rmdefault}{\mddefault}{\updefault}{\color[rgb]{0,0,0}$b$}%
}}}}
\put( 76,-286){\makebox(0,0)[lb]{\smash{{\SetFigFont{12}{14.4}{\rmdefault}{\mddefault}{\updefault}{\color[rgb]{0,0,0}$\EuScript B$}%
}}}}
\put(4126,-211){\makebox(0,0)[lb]{\smash{{\SetFigFont{12}{14.4}{\rmdefault}{\mddefault}{\updefault}{\color[rgb]{0,0,0}$e$}%
}}}}
\put(4126,614){\makebox(0,0)[lb]{\smash{{\SetFigFont{12}{14.4}{\rmdefault}{\mddefault}{\updefault}{\color[rgb]{0,0,0}$b$}%
}}}}
\put(4126,239){\makebox(0,0)[lb]{\smash{{\SetFigFont{12}{14.4}{\rmdefault}{\mddefault}{\updefault}{\color[rgb]{0,0,0}$\EuScript B$}%
}}}}
\put(376,-709){\makebox(0,0)[lb]{\smash{{\SetFigFont{12}{14.4}{\rmdefault}{\mddefault}{\updefault}{\color[rgb]{0,0,0}$e$}%
}}}}
\put(3601,689){\makebox(0,0)[lb]{\smash{{\SetFigFont{12}{14.4}{\rmdefault}{\mddefault}{\updefault}{\color[rgb]{0,0,0}$a$}%
}}}}
\put(676,-1069){\makebox(0,0)[lb]{\smash{{\SetFigFont{12}{14.4}{\rmdefault}{\mddefault}{\updefault}{\color[rgb]{0,0,0}$D$}%
}}}}
\put(3526,-361){\makebox(0,0)[lb]{\smash{{\SetFigFont{12}{14.4}{\rmdefault}{\mddefault}{\updefault}{\color[rgb]{0,0,0}$D$}%
}}}}
\put(2926,479){\makebox(0,0)[lb]{\smash{{\SetFigFont{12}{14.4}{\rmdefault}{\mddefault}{\updefault}{\color[rgb]{0,0,0}$W^u(Q)$}%
}}}}
\end{picture}%

\end{center}
\caption{Beaks.}\label{4_beaks}
\end{figure}

\medskip

{\bfseries Case A. $D\stackrel{\mathrm{def}}{=}dist(\EuScript B, f^n(R))\ge|y_1-y_3|$}.
In this case according to the estimates above we have that $[a,e]$ is tilted at least by $D$ while $[b,e]$ is tilted at most by $D^2$. Hence
$$
|x_1-x_2|\ge|y_1-y_3|(D-D^2)\ge\frac12|y_1-y_3|D\ge\frac12|y_1-y_3|^2.
$$

\medskip

{\bfseries Case B.} We allow $|y_1-y_3|$ to be greater than $dist(\EuScript B, f^n(R))$. Also we make a simplifying assumption that $x_3=0$ and $y_3>0$. Then
\begin{multline*}
|x_1-x_3|\ge\int_{y_3}^{y_1}\sqrt{x^2+y^2}d\,{length}-\int_{y_3}^{y_1}\frac{x^2}2d\,{length}\\
\ge\int_{y_3}^{y_1}(y-y^2/2)d\,{length}\ge\frac13|y_1-y_3|^2.
\end{multline*}

When $x_3\neq 0$ the estimate is similar but first one needs to ``cut" the ``tip of the beak" where $x^2\le y^2$ does not hold. When $y_3<0$ while $y_1>0$ or vice versa modification in the same spirit is required.
\end{proof}

\section{A positive result}

Here we prove Theorem~\ref{fisher}.

Let $M$ be a Riemannian manifold and $d(\cdot,\cdot)$ the distance
function induced by the Riemannian metric. Let $f\colon M\to M$ be
an Anosov diffeomorphism and $g \colon M\to M$ a diffeomorphism
H\"older conjugate to $\tilde f$:
$$
h\circ\tilde f=f\circ h.
$$
Let $\alpha$ be H\"older exponent of $h$ and $\beta$ the H\"older
exponents of $h^{-1}$. Recall that we assume that $\alpha\beta>1/2$.

\begin{definition}
A sequence of points $\{y_i\in M; i\in\mathbb Z\}$ is called
$\varepsilon$-pseudo orbit for $g\colon M\to M$ if
$d(g(y_i),y_{i+1})<\varepsilon$, $i\in\mathbb Z$.
\end{definition}
\begin{definition}
We say that a real orbit $\{f^i(x); i\in\mathbb Z\}$
$\delta$-shadows an $\varepsilon$-pseudo orbit $\{y_i; i\in\mathbb
Z\}$ if $d(f^i(x),y_i)<\delta$.
\end{definition}
\begin{definition}
Diffeomorphism $g\colon M\to M$ is {\it quasi-Anosov} if for all
non-zero $v\in TM$ the sequence $\{\|T^iv\|; i\in\mathbb Z\}$ is
unbounded.
\end{definition}

We will be using the following characterization of Anosov systems.

\begin{theorem}[e.~g.~\cite{M}]
Diffeomorphism $g\colon M\to M$ is Anosov if and only if $g$ is
quasi-Anosov and all dimensions of stable manifolds at periodic
points are the same.
\end{theorem}
Dimensions of stable manifolds at periodic points of $\tilde f$ are
the same since $\tilde f$ is topologically conjugate to $f$. Hence
we only need to show that $\tilde f$ is quasi-Anosov.

Assume that $\tilde f$ is not quasi-Anosov. Then $\exists v\in
T_{\tilde x}M$, $\|v\|=1$, such that $\|T^nv\|\le 1$ for all
$n\in\mathbb Z$. Define $v_n=T^nv$, $n\in\mathbb Z$. For any
sufficiently small $\varepsilon>0$ consider sequence $\{\tilde
x_n=exp(\varepsilon v_n); n\in\mathbb Z\}$. Sequence $\{\varepsilon
v_n\}$ being $\varepsilon$-small and diffeomorphism being
$C^{1+Lip}$ imply that there exists a constant $\tilde c$ that
depends on $\tilde f$ only such that $\{\tilde x_n\}$ is $\tilde
c\varepsilon^2$-pseudo orbit which is obviously $\delta$-shadowed by
the orbit of $\tilde x$ with $\varepsilon/2<\delta<2\varepsilon$.

Let $x=h(\tilde x)$ and $x_n=h(\tilde x_n)$, $n\in\mathbb Z$.
Applying H\"older inequalities we get that $\{x_n\}$ is
$c_1\varepsilon^{2\alpha}$-pseudo orbit for $f$ that is
$\delta$-shadowed by $\{f^n(x)\}$ with
$\delta>c_2\varepsilon^{1/\beta}$. Constants $c_1$ and $c_2$ do not
depend on $\varepsilon$. To make it more transparent denote
$\xi=c_1\varepsilon^{2\alpha}$.

For arbitrarily small $\xi>0$ we have constructed a $\xi$-pseudo
orbit and a true orbit that $\delta$-shadows the pseudo orbit with
$\delta>c\xi^\kappa\stackrel{\mathrm{def}}{=}\xi^{1/2\alpha\beta}$
where $\kappa<1$ by the assumption. Meanwhile it is a well known
simple fact that $\delta$ can be estimated from above $\delta<C\xi$
where $C$ depends only on $f$. The proof is straightforward and
exploits local product structure of stable and unstable foliation of
$f$. For $\xi$ small enough these bounds on $\delta$ contradict each
other. Hence we have arrived at a contradiction.

\end{document}